\documentclass[11pt]{amsart}

\usepackage{amsmath,amssymb,latexsym}
\usepackage{dsfont}
\usepackage[T1]{fontenc}
\usepackage{enumerate}
\usepackage{eqnarray}
\usepackage{mathtools}
\usepackage{url}
\usepackage{xcolor}

\usepackage[a4paper,top=3cm, bottom=3cm, inner=3.5cm,outer=2.7cm,nofoot,headsep=1.5cm]{geometry}
\def\Ind#1#2{#1\setbox0=\hbox{$#1x$}\kern\wd0\hbox to 0pt{\hss$#1\mid$\hss}
\lower.9\ht0\hbox to 0pt{\hss$#1\smile$\hss}\kern\wd0}
\def\Notind#1#2{#1\setbox0=\hbox{$#1x$}\kern\wd0\hbox to 0pt{\mathchardef
\nn="3236\hss$#1\nn$\kern1.4\wd0\hss}\hbox to 0pt{\hss$#1\mid$\hss}\lower.9\ht0
\hbox to 0pt{\hss$#1\smile$\hss}\kern\wd0}

\theoremstyle{plain}
\newtheorem{theorem}{Theorem}[section]
\newtheorem*{theorem*}{Theorem}
\newtheorem{prop}[theorem]{Proposition}
\newtheorem{fact}[theorem]{Fact}
\newtheorem{lemma}[theorem]{Lemma}

\theoremstyle{definition}
\newtheorem{defn}[theorem]{Definition}
\newtheorem{definition}[theorem]{Definition}
\newtheorem{remark}[theorem]{Remark}

\newtheorem{problem}[theorem]{Problem}

\newcommand{\tp}{\operatorname{tp}}

\newcommand{\alg}{\operatorname{alg}}
\newcommand{\Th}{\operatorname{Th}}

\newcommand{\eq}{\operatorname{eq}}
\newcommand{\sscl}{\operatorname{sscl}}
\newcommand{\ch}{\operatorname{char}}
\newcommand{\defeq}{:=}

\newcommand{\supp}{\operatorname{supp}}

\title{Combinatorics in one-based and related structures}
\author{Artem Chernikov and Sergei Starchenko}

\begin{document}

\maketitle
\begin{abstract}
We consider some extremal combinatorial questions for bipartite graphs definable in one-based and related structures. We show that they satisfy both strong Erd\H{o}s-Hajnal property and linear Zarankiewicz. We also show that the same is true for  both collapsed and uncollapsed Hrushovski's ``ab initio'' constructions, and discuss some connections to Zilber's trichotomy principle. For strong Erd\H{o}s-Hajnal, we show that in fact it holds in a more general class of $1$-semi-equational theories.
\end{abstract}

\section{Introductions}
This note contains some results on combinatorics of definable relations in stable one-based theories made by the authors mostly prior to and during the 2018 Institut Henri Poincar\'e trimester ``Model Theory, Combinatorics and Valued fields'' and presented on several occasions since then (starting with \cite{TalkPanhel} for strong Erd\H{o}s-Hajnal and \cite{TalkBedlewo} for linear Zarankiewicz). As well as a more recent result on the so-called $1$-semiequational theories (a generalization of $1$-based from stability to NIP proposed in \cite{chernikov2023semi}). One-basedness is a central model theoretic notion of ``linearity'', and \emph{one based} theories form a subclass of stable theories with the behavior of forking  independence similar to vector spaces or modules (as opposed to fields), corresponding to local modularity in the strongly minimal case. 

\begin{theorem*}[Theorem \ref{thm: sEH in 1based}]
	Every relation definable in a one-based theory $T$ satisfies the strong Erd\H{o}s-Hajnal property (Definition \ref{def: sEH}), and moreover admits a (not necessarily definable in $T$) distal cell decomposition (Definition \ref{def: distal cell decomp}).
\end{theorem*}

In fact, using a recent result \cite{scott2023pure} in structural graph theory, we also prove a more general result:
\begin{theorem*}[Theorem \ref{thm: 1-semieq sEH}]
	Every relation definable in a $1$-semi-equational theory $T$ (Definition \ref{defn: 1-semieq}) satisfies the strong Erd\H{o}s-Hajnal property.
\end{theorem*}

\begin{theorem*}[Theorem \ref{thm: bool comb weak norm is comb lin}]
	Every relation definable in a one-based theory satisfies linear Zarankiewicz (Definition \ref{def: lin Zarank}).
\end{theorem*}

\begin{theorem*}[Theorem \ref{thm: ab init lin Zar and sEH}]
	Let $T$ be the theory of a Hrushovski's collapsed or uncollapsed ``ab initio'' structure (it is CM-trivial, but not one based). Then every relation definable in $T$ satisfies both  strong Erd\H{o}s-Hajnal and linear Zarankiewicz.
\end{theorem*}

This contributes to the emerging connection between model theoretic linearity and combinatorial linearity, pointing that the geometry of forking independence and presence of algebraic structure (Zilber's trichotomy principle and its variants) is reflected by certain asymptotic finite counting properties in sufficiently tame classes of theories.
 In particular, Theorem \ref{thm: bool comb weak norm is comb lin} strengthens  the ``sparse pseudoplanes'' property from \cite[Proposition 3.1]{evans2005trivial} (which gives an analogous conclusion, but assuming nfcp and that the relation is $K_{t,t}$-free globally).  Almost linear Zarankiewicz for semi-linear hypergraphs, and more generally for Boolean combinations of basic relations (Definition \ref{def: basic rels}),  is proved in \cite{basit2021zarankiewicz}; it it also shown there that for $o$-minimal theories, almost linear Zarankiewicz is equivalent to local modularity. Strong Erd\H{o}s-Hajnal for graphs with bounded VC-minimal complexity is established in \cite{fu2023note} (Theorem \ref{thm: 1-semieq sEH} generalizes this).  It is demonstrated in \cite{basit2025shatter} that semilinear set systems have integer valued VC-density. The global version of (almost) linear Zarankiewicz in $o$-minimal and related contexts is further studied in \cite{eleftheriou2025global}. Further evidence that almost linear Zarankiewicz can be viewed as a model theoretic notion of linearity in various contexts is provided in \cite{walsberg2026trace}.  A possible generalization of weakly normal formulas from stability to NIP, the so called 1-semi-equations, is proposed and studied in \cite{chernikov2023semi}. One-basedeness and local modularity also play a key role in the model theoretic approach to Elekes-Szab\'o-type theorems \cite{hrushovski2013pseudo, chernikov2021model, zbMATH07452921, chernikov2024model}.

We conclude with some open questions and research directions.

\section{Strong Erd\H{o}s-Hajnal in one-based theories}

\subsection{One based theories and weakly normal formulas}
Most relevant for us here, one based theories admit a  combinatorial characterization in terms of the so-called \emph{weakly normal} formulas (\cite{hrushovski1987weakly}, see also \cite{pillay1996geometric}). We let $M$ be a model of a complete theory $T$.

\begin{defn}
	\begin{enumerate}
		\item We say that a partitioned formula $\varphi(x,y)$ is \emph{$k$-weakly-normal} if for any $b_1, \ldots, b_k \in M^{y}$ with $\bigcap_{i=1}^k \varphi(M,b_i) \neq \emptyset$  we have $\varphi(M,b_i) = \varphi(M,b_j)$ for some $i \neq j \in [k]$.
		\item $\varphi(x,y)$ is \emph{weakly normal} if it is $k$-weakly normal for some $k \in \mathbb{N}$.
		\item $\varphi(x,y)$ is \emph{normal} if it is $2$-weakly normal.
	\end{enumerate}
\end{defn}
\begin{fact}\label{fac: 1-based iff bool comb weakly normal}\cite{hrushovski1987weakly}
	A stable theory $T$ is $1$-based if and only if every partitioned formula $\varphi(x;y)$, with $x,y$ arbitrary tuples of variables, is equivalent to a Boolean combination of weakly normal formulas. If $T$ is an expansion of a group, it is $1$-based if and only if every partitioned formula $\varphi(x,y)$ is a Boolean combination of  normal formulas.
\end{fact}

\subsection{Strong EH}

Let a relation $E\subseteq X\times Y$ be given.
\begin{defn}\label{def: sEH}
We say that $E\subseteq X\times Y$ satisfies the\emph{ strong EH}
property, or sEH, if there is some real $\beta>0$ such that for any
probability measure $\mu$ on $X$ and $\nu$ on $Y$ concentrated
on finite subsets there are some $X_{0}\subseteq X,Y_{0}\subseteq Y$
such that $\mu\left(X_{0}\right),\nu\left(Y_{0}\right)>\beta$ and
either $X_{0}\times Y_{0}\subseteq E$ or $\left(X_{0}\times Y_{0}\right)\cap E=\emptyset$.
\end{defn}

\begin{defn}
We say that $E\subseteq X\times Y$ satisfies the \emph{density strong
EH} property, or dsEH, if for any real $\alpha>0$ there is some real
$\beta>0$ such that for any probability measure $\mu$ on $X$ and
$\nu$ on $Y$ concentrated on finite subsets the following are satisfied (here $\otimes$ denotes the product measure):

\begin{enumerate}
\item if $\mu\otimes\nu\left(E\right)>\alpha$, then there are some $X_{0}\subseteq X,Y_{0}\subseteq Y$
such that $\mu\left(X_{0}\right),\nu\left(Y_{0}\right)>\beta$ and
$X_{0}\times Y_{0}\subseteq E$;
\item if $\mu\otimes\nu\left(E\right)<1-\alpha$, then there are some $X_{0}\subseteq X,Y_{0}\subseteq Y$
such that $\mu\left(X_{0}\right),\nu\left(Y_{0}\right)>\beta$ and
$\left(X_{0}\times Y_{0}\right)\cap E=\emptyset$.
\end{enumerate}
\end{defn}

\begin{remark}
\label{rem:=000020dseh=000020closed=000020under=000020boolean=000020combs}

\begin{enumerate}
\item If $E$ satisfies dsEH, then it satisfies sEH.
\item The family of all relations $E\subseteq X\times Y$ satisfying sEH
is closed under boolean combinations.
\item If $E$ satisfies sEH, then it satisfies dsEH as well.
\item The family of all relations $E\subseteq X\times Y$ satisfying dsEH
is closed under boolean combinations.
\end{enumerate}
\end{remark}

\begin{proof}
(1) is obvious by taking $\beta$ given by dsEH with $\alpha=\frac{1}{3}$.

(2) is clear from the definition. Indeed, assume we are given $E \subseteq X \times Y $ a Boolean combination of $E_0$ and $E_1$, satisfying sEH with $\alpha_0$ and $\alpha_1$ respectively,  and finitely supported measures $\mu, \nu$. First find an $E_0$-homogeneous   pair of sets $X_0, Y_0$ with $\mu(X_0), \nu(Y_0) >  \alpha_0$. Define $\mu'$ by localizing $\mu$ on $X_0$, via $\mu'(Z) := \mu(Z \cap X_0)/\mu(X_0)$, and similarly for $\nu'$ by localizing $\nu$ on $Y_0$. Then there exist sets $X_1 \subseteq X, Y_1 \subseteq Y$ with $\mu'(X_1), \nu'(Y_1) > \alpha_1$ so that $E_1$ is homogeneous on $X_1 \times Y_1$. It follows that $E$ is homogeneous on $(X_0 \cap X_1) \times (Y_0 \cap Y_1)$ and $\mu(X_0 \cap X_1), \nu(Y_0 \cap Y_1) > \alpha := \alpha_0 \alpha_1$.

(3) is less obvious and follows from the proof of \cite[Proposition 4.1]{chernikov2018regularity}
(ignoring the definability clause).

(4) follows by combining (1), (2) and (3).
\end{proof}

A higher arity version of the strong EH property is considered in \cite{chernikov2026n} in connection to $n$-distality.

\subsection{Distality}

We recall some terminology and facts around local distality (i.e.~the distality assumption is only made for a particular formula as opposed for the whole theory).

\begin{defn}\label{def: distal cell decomp}\cite[Section 2]{chernikov2020cutting} (see also \cite[Section 2]{chernikov2024model} for a discussion) Let $X,Y$ be infinite sets, and $E\subseteq X\times Y$ a binary relation.
 \begin{enumerate}
 	\item Let $A\subseteq X$. For $b\in Y$, we say that $E_{b} = \left\{a \in X : (a,b) \in E \right\}$ \emph{crosses} $A$ if $E_{b}\cap A\neq\emptyset$ and $\left(X\setminus E_{b}\right)\cap A\neq\emptyset$.
 	\item A set $A\subseteq X$ is \emph{$E$-complete over $B\subseteq Y$} if $A$ is not crossed by any $E_{b}$ with $b\in B$.
 	\item A family $\mathcal{F}$ of subsets of $X$ is a \emph{cell decomposition for $E$ over $B\subseteq Y$} if $X\subseteq\bigcup\mathcal{F}$ and every $A\in\mathcal{F}$ is $E$-complete over $B$.
 	\item A \emph{cell decomposition for $E$} is a map $\mathcal{T}: B \mapsto \mathcal{T}(B)$ such that for each finite $B\subseteq Y$, $\mathcal{T}\left(B\right)$ is a cell decomposition for $E$ over $B$.
 	\item A cell decomposition $\mathcal{T}$ is \emph{distal} if there exist $k\in\mathbb{N}$ and a relation $D\subseteq X\times Y^{k}$ such that for all finite $B\subseteq Y$, 
 	$$\mathcal{T}\left(B\right)=\{D_{\left(b_{1},\ldots,b_{k}\right)}:b_{1},\ldots,b_{k}\in B\mbox{ and }D_{\left(b_{1},\ldots,b_{k}\right)}\mbox{ is }E\mbox{-complete over }B\}.$$
 	\item For $t \in \mathbb{R}_{>0}$, we say that a cell decomposition $\mathcal{T}$ has \emph{exponent $\leq t$} if there exists some $c \in \mathbb{R}_{>0}$ such that $|\mathcal{T}(B)| \leq c |B|^{t}$ for all finite sets $B \subseteq Y$.
 \end{enumerate}
 \end{defn}

 \begin{remark}
 	Note that if $\mathcal{T}$ is a distal cell decomposition, then it has exponent $\leq k$ for $k$ as in Definition \ref{def: distal cell decomp}(5).
 \end{remark}
 
 \begin{remark}\label{rem: dist cell decomp Bool comb}
 	The class of relations $E \subseteq X \times Y$ admitting distal cell decompositions is closed under Boolean combinations (by taking common refinements of the decompositions).
 \end{remark}
 
 Existence of ``strong honest definitions'' established in \cite{chernikov2015externally} shows that every relation definable in a distal structure admits a distal cell decomposition (of some exponent).

\begin{fact}\label{fac: def in dis impl distal cell decomp}(see \cite[Fact 2.9]{chernikov2020cutting})
Assume that the relation $E$ is definable in a distal structure $\mathcal{M}$. Then $E$ admits a distal cell decomposition (of some exponent $t \in \mathbb{N}$).
Moreover, in this case the relation $D$ in Definition \ref{def: distal cell decomp}(5) is also definable in $\mathcal{M}$ (and distality of $\mathcal{M}$ is equivalent to this property being satisfied for all formulas by \cite{chernikov2018regularity}).
\end{fact}

The following definition abstracts from the notion of cuttings in incidence geometry (see the introduction of \cite{chernikov2020cutting} for an extended discussion).
\begin{defn}\label{def: r-cutting}
Let $X,Y$ be infinite sets, $E\subseteq X\times Y$.
We say that
  $E$ \emph{admits cuttings with exponent $t \in \mathbb{R}$} if there is some constant $c \in \mathbb{R}_{>0}$ satisfying the following. For any $B \subseteq Y$ with $\left| B \right| = n$ and any $r \in \mathbb{R}$ with $1 < r < n$ there are some sets $X_1, \ldots, X_s \subseteq X$ covering $X$ with $s \leq c r ^t$ and such that for each $i \in [s]$ there are at most $\frac{n}{r}$ elements $b \in B$ so that  $X_i$ is crossed by $E_b$.
  \end{defn}

In the case $r > n$ in Definition \ref{def: r-cutting}, sets in the covering are not crossed at all. And for $r$ varying between $1$ and $n$, $r$-cutting allows to control the trade-off between the number of cells in a covering and the number of times each cell is allowed to be crossed.

\begin{fact}(Distal cutting lemma, \cite[Theorem 3.2]{chernikov2020cutting})\label{fac: distal implies cutting}
Assume $E\subseteq X\times Y$ admits a distal cell decomposition $\mathcal{T}$ of exponent $ \leq t$. Then $E$ admits cuttings with exponent $ \leq t$. Moreover, every set in this cutting is an intersection of at most two cells in $\mathcal{T}$.
\end{fact}

\begin{prop}\label{prop: cutting implies sEH}
Assume $E \subseteq X \times Y$ admits cuttings with exponent $t$, for some $t \in \mathbb{N}$ (by Fact \ref{fac: distal implies cutting} this holds if $E$ admits a distal cell decomposition). Then  $E$ satisfies strong EH.
\end{prop}

\begin{proof}

Let $c,t$ be given by Definition \ref{def: r-cutting} for $E$ (we will only use it with $r=2$). We show that $E$ satisfies sEH with any $\delta < \min \{\frac{1}{c 2^t}, \frac{1}{4}\}$.

Let $\mu, \nu$ be finitely supported probability measures on $X,Y$ respectively, and assume first that $\nu$ is the uniform measure on a finite set $B \subseteq Y$ of size $n$, and without loss of generality $n \geq 3$. By assumption with $r=2$, there exist some sets $X_1, \ldots, X_s \subseteq X$ covering $X$ so that $s \leq c 2^t$, and for each $i \in [s]$ there are at most $\frac{n}{2}$ element $b \in B$ so that $X_i$ is crossed by $E_b$. In particular, $\mu(X_{i^{\ast}}) \geq \frac{1}{c 2^t}$ for at least one $i^{\ast} \in [s]$. And there is a set $B_0 \subseteq B$ with $|B_0| \geq \frac{1}{2}(n - \frac{n}{2}) \geq n/4$ so that either $X_i \subseteq E_b$ for all $b \in B_0$, or $X_i \subseteq E \setminus E_b$ for all $b \in B_0$. In either case $(X_{i^{\ast}}, B_0)$ is an $E$-homogeneous pair with $\mu(X_{i^{\ast}}) \geq \frac{1}{c 2^t}$ and $\nu(B_0) \geq |B_0|/|B| \geq \frac{1}{4}$. For a general finitely supported measure $\nu$, we may first assume that the weights of the points are rational, replacing the weights one by one by a rational number very close to the original weight (with respect to our fixed $\delta$ and the size of the support), and then reduce to the uniform measure by duplicating each point in the support according to its weight (this clearly does not affect the existence of a distal cell decomposition with given size).
\end{proof}

In particular, every reduct of a distal structure admits (not necessarily definable) distal cell decompositions, hence satisfies strong EH. We note that the original proof of the sEH in distal theories also used that the formula defining the distal cell decomposition is itself NIP, but it is not necessary.

\subsection{Strong Erd\H{o}s-Hajnal holds in one-based theories}
In the rest of the section we prove the following theorem.
\begin{theorem}\label{thm: sEH in 1based}
Let $T$ be stable, $1$-based. Then every definable relation $\varphi(x,y)$, with $x,y$ arbitrary finite tuples, admits a (not necessarily definable) distal cell decomposition and satisfies dsEH (hence also sEH).
\end{theorem}

We give two proofs of this. In the first one, we show that given a weakly normal formula we can expand $M$ by a  linear order to get a distal cell decomposition for it, hence the result follows by Proposition \ref{prop: cutting implies sEH}.  In the second proof, we show that any $k$-weakly normal formula can be modeled by the intersections of affine hyperplanes in $\mathbb{R}^k$, i.e.~by a semialgebraic relation, hence satisfies sEH already by \cite{alon2005crossing}.

\subsection{First proof}

\begin{prop}\label{prop: weak norm implies strong honest def}
	Every weakly normal formula admits a (not necessarily definable) distal cell decomposition. Hence also every Boolean combination of weakly normal formulas admits a distal cell decomposition (by Remark \ref{rem: dist cell decomp Bool comb}).
\end{prop}
\begin{proof}
	Assume $\varphi(x,y)$ is $k$-weakly normal. Consider the definable equivalence relation $E_{\varphi}(y,y') :\iff  \forall x (\varphi(x,y) \leftrightarrow \varphi(x,y'))$. Choose any linear order $\leq_{\varphi}$ on $M^y$ so that each class of $E_{\varphi}$ is convex, with a minimal and a maximal element (which may be equal). As $\varphi(x,y)$ is $k$-weakly normal, for every $a \in M^x$ the set $\varphi(a,M^y)$ is given by the union of at most $k$ many classes of $E_{\varphi}$, hence by the union of at most $k$ many closed intervals in $(M^y, <_{\varphi})$. For $t \in [k]$, consider the functions $f_t^{\ell}, f_t^{r}: M^x \to M^y$ so that $f_t^{\ell}(a), f_t^{r}(a)$ are the left/right endpoints of the $t$'s interval of $\varphi(a,M^y)$ (or of the maximal $t'<t$ interval if there are fewer than $t$ intervals; if $\varphi(a,M^y)$ is empty, let these be an arbitrary fixed point). 
	
		Now, given an arbitrary finite set $B \subseteq M^y$ and $a \in M^x$, for each $t \in [k]$ we let $b^{\ell}_t$ be the minimal element of $B$ with $f_t^{\ell}(a) \leq  b^{\ell}_t$ and $b^{r}_t$ the maximal element of $B$ with $b^{r}_t \leq f_t^{r}(a)$ (if they exist; if not pick an arbitrary point in $B$). We also let $c^{\ell}_t$ be the maximal element of $B$ with $c^{\ell}_t < f_t^{\ell}(a)$, and $c^{r}_t$ the minimal element of $B$ with $f_t^{r}(a) <  c^{r}_t$.
	Then the formula
	\begin{gather*}
		D \left( x; (b^\ell_t)_{t \in [k]}, (b^r_t)_{t \in [k]},  (c^\ell_t)_{t \in [k]}, (c^r_t)_{t \in [k]} \right) := \\
		\bigwedge_{t = 1}^{k} \left( f^{\ell}_t(x)  \leq  b^{\ell}_t \land f^{r}_t(x) \geq b^r_t\right) \land \left( c^{\ell}_t < f^{\ell}_t(x) \land f^{r}_t(x) <  c^{r}_t\right)
	\end{gather*}
	is satisfied by $a$, and for any other $a' \in M^x$ satisfying it, $\tp_{ \varphi}(a'/B) = \tp_{ \varphi} (a/B)$. Hence $D \left( x; (y^\ell_t)_{t \in [k]}, (y^r_t)_{t \in [k]},  (z^\ell_t)_{t \in [k]}, (z^r_t)_{t \in [k]} \right)$ gives a distal cell decomposition for $\varphi$.
\end{proof}

It follows that every weakly normal formula satisfies sEH by Proposition \ref{fac: distal implies cutting}, hence every formula in a one-based theory satisfies sEH by Remark \ref{rem:=000020dseh=000020closed=000020under=000020boolean=000020combs}.

\subsection{Second proof}

Let now $E\subseteq X\times Y$ be given. We associate each element
$a\in X$ with a function $f_{a}:Y\to\left\{ 0,1\right\} $ defined
by 
\[
f_{a}\left(b\right)=\begin{cases}
1 & \mbox{if }E\left(a,b\right)\mbox{ holds}\mbox{,}\\
0 & \mbox{if }\neg E\left(a,b\right)\mbox{ holds.}
\end{cases}
\]

Thus for every finite set $B\subseteq Y$ we have a set of functions
$\Omega_{E}\subseteq\left\{ 0,1\right\} ^{B}$ defined by 
\[
\Omega_{E}\left(B\right)=\left\{ f_{a}\restriction B:a\in X\right\} \mbox{.}
\]

More generally, let $\Omega$ be any assignment that assigns to
any finite $B\subseteq Y$ a set of functions $\Omega\left(B\right)$
from $B$ to $\left\{ 0,1\right\} $. Let $B\subseteq Y$ be finite.
\begin{defn}
Given a probability measure $\nu$ on $B$ and a probability measure
$\mu$ on $\Omega\left(B\right)$ we define the \emph{density of $\Omega$
on $B$} as 
\[
d_{\nu,\mu}^{\Omega}\left(B\right)=\sum_{f\in\Omega\left(B\right)}\mu\left(f\right)\left(\sum_{b\in B}f\left(b\right)\nu\left(b\right)\right)\mbox{.}
\]
\end{defn}

\begin{defn}
We say that an assignment $\Omega:B\mapsto\Omega\left(B\right)$
has the \emph{density strong EH, }or\emph{ dsEH, }if for every real
$\alpha>0$ there is some real $\beta>0$ such that for all finite
$B\subseteq Y$, a probability measure $\nu$ on $B$ and a probability
measure $\mu$ on $\Omega\left(B\right)$ the following are satisfied:

\begin{itemize}
\item if $d_{\nu,\mu}^{\Omega}\left(B\right)>\alpha$, then there are
$B_{0}\subseteq B$ and $A_{0}\subseteq\Omega\left(B\right)$ such
that $\mu\left(A_{0}\right),\nu\left(B_{0}\right)>\beta$ and $f\left(b\right)=1$
for all $f\in A_{0}$ and $b\in B_{0}$.
\item if $d_{\nu,\mu}^{\Omega}\left(B\right)<1-\alpha$, then there are
$B_{0}\subseteq B$ and $A_{0}\subseteq\Omega\left(B\right)$ such
that $\mu\left(A_{0}\right),\nu\left(B_{0}\right)>\beta$ and $f\left(b\right)=0$
for all $f\in A_{0}$ and $b\in B_{0}$.
\end{itemize}
\end{defn}

\begin{lemma}
\label{lem:=000020DSEH=000020transfer}A relation $E\subseteq X\times Y$
satisfies dsEH if and only if $\Omega_{E}$ does.
\end{lemma}

\begin{proof}
Let $B$ be the support of $\nu$. Given a finitely supported measure
$\mu$ on $X$, we associate to it a measure $\mu'$ on $\Omega_{E}\left(B\right)$
by taking $\mu'\left(f\right):=\sum\left\{ \mu\left(a\right):a\in X,f_{a}\restriction B=f\right\} $.
And conversely, given a probability measure $\mu'$ on $\Omega_{E}\left(B\right)$,
for each $f\in\Omega_{E}\left(B\right)$ choose some $a_{f}\in X$
such that $f=f_{a_{f}}\restriction B$, and define a probability measure
$\mu$ on $A=\left\{ a_{f}:f\in\Omega_{E}\left(B\right)\right\} \subseteq X$
by $\mu\left(a_{f}\right):=\mu'\left(f\right)$. Then $d_{\nu,\mu'}^{\Omega}\left(B\right)=\mu\otimes\nu\left(E\right)$,
and the rest is straightforward from the definitions.
\end{proof}
\begin{lemma}
\label{lem:=000020DSEH=000020for=000020subassignments}Let $\Omega$
and $\Omega'$ be two assignments such that $\Omega\left(B\right)\subseteq\Omega'\left(B\right)$
for all finite $B\subseteq Y$. If $\Omega'$ satisfies dsEH, then
so does $\Omega$.
\end{lemma}

\begin{proof}
Let $\alpha>0$ be given, and let $\beta=\beta\left(\alpha\right)$
be as given by dsEH for $\Omega'$. Let $B\subseteq Y$ be a finite
set, $\mu$ a probability measure on $\Omega\left(B\right)$ and $\nu$
a probability measure on $B$ such that $d_{\nu,\mu}^{\Omega}\left(B\right)>\alpha$.
By hypothesis $\Omega\left(B\right)\subseteq\Omega'\left(B\right)$,
so we can define a probability measure $\mu'$ on $\Omega'\left(B\right)$
by taking $\mu'\left(X\right):=\mu\left(X\cap\Omega\left(B\right)\right)$
for all $X\subseteq\Omega'\left(B\right)$. Then $d_{\nu,\mu'}^{\Omega'}\left(B\right)=d_{\nu,\mu}^{\Omega}\left(B\right)>\alpha$,
and by dsEH for $\Omega'$ we can find some $B_{0}\subseteq B$ and
$A_{0}\subseteq\Omega'\left(B\right)$ with $\mu'\left(A_{0}\right),\nu\left(B_{0}\right)>\beta$
and $f\left(b\right)=1$ for all $f\in A_{0}$ and $b\in B_{0}$.
But by the choice of $\mu'$, replacing $A_0$ by $A_0 \cap \Omega(B)$, we have   $A_{0}\subseteq\Omega\left(B\right)$
and $\mu\left(A_{0}\right)>\beta$, as wanted.

The case $d_{\nu,\mu}^{\Omega}\left(B\right)<1-\alpha$ is analogous.
\end{proof}
Say that an assignment $\Omega$ is \emph{$k$-bounded }if for any
finite $B\subseteq Y$ and any $f\in\Omega\left(B\right)$ we have
$\sum_{b\in B}f\left(b\right)\leq k$.
\begin{lemma}
\label{lem:=000020DSEH=000020for=000020k-bounded=000020assignments}For
any $k\in\mathbb{N}$, if $\Omega$ is $k$-bounded then it satisfies
dsEH.
\end{lemma}

\begin{proof}
Let $\Omega_{k}$ be the \emph{universal $k$-bounded assignment}
on $Y$. Namely, for any finite $B\subseteq Y$ we take 
\[
\Omega_{k}\left(B\right):=\left\{ f\in\left\{ 0,1\right\} ^{B}:\sum_{b\in B}f\left(b\right)\leq k\right\} .
\]

By Lemma \ref{lem:=000020DSEH=000020for=000020subassignments}, it
is enough to show that $\Omega_{k}$ satisfies dsEH. Fix $k$, and
consider the family $\Sigma$ of all affine hyperplanes in $\mathbb{R}^{k}$.
Then $\Sigma$ is a definable family in the field of reals. For $B\subseteq\Sigma$
and $a\in\mathbb{R}^{k}$, let $f_{a}:B\to\left\{ 0,1\right\} $ be
the incidence function, i.e. $f_{a}\left(b\right)=1$ if and only
if $a\in b$.

The following is by basic linear algebra: if $B\subseteq\Sigma$ is
a finite set of affine hyperplanes in a general position, then $\Omega\left(B\right)=\left\{ f_{a}\restriction B:a\in\mathbb{R}^{k}\right\} $
is equal to $\Omega_{k}\left(B\right)$. Hence we can interpret $\Omega_{k}$
as $\Omega_{E}$, where $E=E'\cap\left(\mathbb{R}^{k}\times Y\right)$,
$Y$ is an infinite set of affine hyperplanes in a general position
in $\mathbb{R}^{k}$ and $E'$ is a definable relation. Then $E$
satisfies dsEH by \cite{chernikov2018regularity} (in fact, already by \cite{alon2005crossing}),  hence the same holds for $\Omega_{k}$
by Lemma \ref{lem:=000020DSEH=000020transfer}.
\end{proof}
\begin{defn}
Let $E\subseteq X\times Y$ be given.

\begin{enumerate}
\item $E$ is \emph{$k$-weakly normal} if for any $b_{1},\ldots,b_{k}\in Y$,
if $\bigcap_{i=1}^{k}E\left(X,b_{i}\right)\neq\emptyset$, then
$E\left(X,b_{i}\right)=E\left(X,b_{j}\right)$ for some $1\leq i<j\leq k$.
\item $E$ is \emph{reduced }if $E\left(X,b\right)\neq E\left(X,b'\right)$
for any $b\neq b'$ in $Y$.
\end{enumerate}
\end{defn}

\begin{remark}
\label{rem:=000020wn=000020and=000020reduced=000020implies=000020bounded}Note
that if $E$ is $k$-weakly normal and reduced, then $\Omega_{E}$
is $k$-bounded.
\end{remark}

For any $E\subseteq X\times Y$, we define an equivalence relation
$\sim$ on $Y$ by saying $b\sim b'\iff E\left(X,b\right)=E\left(X,b'\right)$.
Let $\tilde{Y}:=Y/\sim$ and let $\tilde{E}\subseteq X\times\tilde{Y}$
be defined by $\tilde{E}\left(x,\tilde{y}\right)\iff E\left(x,y\right)$,
where $\tilde{y}$ denotes the $\sim$-class of $y\in Y$. 
\begin{lemma}
\label{lem:=000020reduction=000020preserves=000020DSEH}$E$ satisfies
dsEH if and only if $\tilde{E}$ does.
\end{lemma}
\begin{proof}
Let $\pi:Y\to\tilde{Y}$ be the quotient map, $\pi(y)=\tilde{y}$.
First assume that $\tilde{E}$ satisfies dsEH. Fix $\alpha>0$, and let $\beta=\beta_{\tilde{E}}(\alpha)$
be given by dsEH for $\tilde{E}$. Let $\mu$ and $\nu$ be measures on $X$ and $Y$, respectively. Let
$\tilde{\nu}:=\pi_{*}\nu$ be the push-forward measure on $\tilde{Y}$, so 
$\tilde{\nu}(\tilde{y})=\sum_{y'\sim y}\nu(y')$. Since
$\tilde{E}(x,\pi(y)) \Leftrightarrow E(x,y)$, we have $\mu\otimes\tilde{\nu}(\tilde{E})=\mu\otimes\nu(E)$.
If $\mu\otimes\nu(E)>\alpha$, then by dsEH for $\tilde{E}$ there are
$X_{0}\subseteq X$ and $\tilde{Y}_{0}\subseteq\tilde{Y}$ such that
$\mu(X_{0}),\tilde{\nu}(\tilde{Y}_{0})>\beta$ and
$X_{0}\times\tilde{Y}_{0}\subseteq\tilde{E}$. Put
$Y_{0}:=\pi^{-1}(\tilde{Y}_{0})$. Then
$\nu(Y_{0})=\tilde{\nu}(\tilde{Y}_{0})>\beta$, and for every
$x\in X_{0}$ and $y\in Y_{0}$ we have
$\pi(y)\in\tilde{Y}_{0}$, hence $\tilde{E}(x,\pi(y))$ and so
$E(x,y)$. Thus $X_{0}\times Y_{0}\subseteq E$. The case
$\mu\otimes\nu(E)<1-\alpha$ is identical.

Conversely, assume that $E$ satisfies dsEH. Fix $\alpha>0$, and let
$\beta=\beta_{E}(\alpha)$ be given by dsEH for $E$. Let $\mu, \tilde{\nu}$ be  measures on $X,\tilde{Y}$, respectively. For every
$\tilde{y}$ in the support of $\tilde{\nu}$, choose a representative
$y\in Y$ with $\pi(y)=\tilde{y}$, and define a measure $\nu$ on $Y$ by $\nu(y)=\tilde{\nu}(\tilde{y})$ 
for $\tilde{y}\in\operatorname{supp}(\tilde{\nu})$, and $\nu(y)=0$
for all other $y\in Y$. Then $\mu\otimes\nu(E)=\mu\otimes\tilde{\nu}(\tilde{E})$.
Suppose first that $\mu\otimes\tilde{\nu}(\tilde{E})>\alpha$. By dsEH
for $E$, there are $X_{0}\subseteq X$ and $Y_{0}\subseteq Y$ such that
$\mu(X_{0}),\nu(Y_{0})>\beta$ and $X_{0}\times Y_{0}\subseteq E$.
Let $\tilde{Y}_{0}:=\pi\bigl(Y_{0}\cap\operatorname{supp}(\nu)\bigr)$. 
Then $\tilde{\nu}(\tilde{Y}_{0})=\nu(Y_{0})>\beta$ and 
$X_{0}\times\tilde{Y}_{0}\subseteq\tilde{E}$. The case
$\mu\otimes\tilde{\nu}(\tilde{E})<1-\alpha$ is the same.
\end{proof}

\begin{prop}
\label{prop:=000020DSEH=000020for=000020weakly=000020normal=000020relations}If
$E$ is $k$-weakly normal, then it satisfies dsEH.
\end{prop}

\begin{proof}
Note that if $E$ is $k$-weakly normal, then $\tilde{E}$ is $k$-weakly
normal and reduced, hence $\Omega_{\tilde{E}}$ is $k$-bounded. The
result follows by Lemmas \ref{lem:=000020DSEH=000020for=000020k-bounded=000020assignments}
and \ref{lem:=000020reduction=000020preserves=000020DSEH}.
\end{proof}

The claim follows by combining Proposition \ref{prop:=000020DSEH=000020for=000020weakly=000020normal=000020relations}
and Remark \ref{rem:=000020dseh=000020closed=000020under=000020boolean=000020combs}(4).

\section{Strong EH in $1$-semiequational theories}

\subsection{$1$-semiequational theories}
The following generalization of weakly normal formulas to unstable setting is proposed in \cite{chernikov2023semi}.

\begin{defn}\label{defn: 1-semieq}
	\begin{enumerate}
		\item We say that a partitioned formula $\varphi(x,y)$ is a \emph{$(k,1)$-semi-equation} if for any $b_1, \ldots, b_k \in M^{y}$ with $\bigcap_{i=1}^k \varphi(M,b_i) \neq \emptyset$  we have $\varphi(M,b_i) \subseteq \varphi(M,b_j)$ for some $i \neq j \in [k]$.
		\item $\varphi(x,y)$ is a \emph{$1$-semi-equation} if it is a $(k,1)$-semi-equation for some $k \in \mathbb{N}$.
		\item A theory $T$ is \emph{$1$-semiequational} if every partitioned formula is equivalent to a Boolean combination of $1$-semiequations.
	\end{enumerate}
\end{defn}
\begin{fact}
\begin{enumerate}
	\item  $\varphi(x,y)$ is weakly normal if and only if it is stable and a $1$-semi-equation \cite[Proposition 2.19]{chernikov2023semi}. In particular, a  stable theory is $1$-based if and only if it is $1$-semiequational.
	\item An $o$-minimal expansion $M$ of an ordered group is $(2,1)$-semi-equational if and only if $M$ is linear, i.e.~satisfies the CF property in the sense of \cite{loveys1993linear} (see \cite[Proposition 3.2]{chernikov2023semi}).
\end{enumerate}

\end{fact}

\subsection{$1$-semi-equational theories satisfy strong EH}

\begin{theorem}\label{thm: 1-semieq sEH}
Every $1$-semi-equation satisfies sEH. In particular, every $1$-semi-equational theory satisfies sEH (by Remark \ref{rem:=000020dseh=000020closed=000020under=000020boolean=000020combs}).
\end{theorem}

\begin{remark}
	For the special case of Boolean combinations of $(2,1)$-semiequations, strong EH is proved directly in \cite{fu2023note}. On the other hand, 	it was observed in \cite[Proposition 2.23]{chernikov2023semi}, similarly to the proof of Proposition \ref{prop: weak norm implies strong honest def}, that every $(2,1)$-semi-equation admits a distal cell decomposition,  which using Proposition \ref{prop: cutting implies sEH} gives another proof of this (see also Problem \ref{prob: 1-semieq dist cell}).
\end{remark}

In the rest of the section we prove this theorem. We will use a recent deep result in structural graph theory from \cite{scott2023pure}. A \emph{bigraph} is a graph together with a distinguished ordered bipartition $(V_1,V_2)$, all of whose edges go between $V_1$ and $V_2$.  An \emph{induced copy} of a bigraph $H=(H_1,H_2)$ inside a bigraph $G=(V_1,V_2)$ is always understood side-respectingly: vertices of $H_i$ are mapped to $V_i$, and both cross-edges and cross-nonedges are preserved.

\begin{fact}(Scott--Seymour--Spirkl, \cite[Theorem 1.7]{scott2023pure})\label{fact:sss}
For every forest bigraph $H$ and every $\tau\in(0,1]$, there is a constant $\varepsilon_H(\tau)>0$ such that every finite induced $H$-free bigraph $G=(V_1,V_2,E)$ satisfying $|E|\leq (1-\tau)|V_1||V_2|$ 
contains an anticomplete pair of sets $Z_i \subseteq V_i$  with $|Z_i|\geq \varepsilon_H(\tau)|V_i|$ for both $i =1,2$.
\end{fact}

We will apply it with particular tree bigraphs:
\begin{definition}
	Fix $k\geq 2$.  Let $T_k$ be the bigraph with left side $L(T_k)=\{c,\ell_1,\ldots,\ell_k\}$, right side $R(T_k)=\{r_1,\ldots,r_k\}$, and the only edges are $(c,r_i)$ and $(\ell_i, r_i)$ for $i=1, \ldots, k$. In particular $T_k$ is a tree (with root $c$).
\end{definition}
\begin{remark}\label{lem:no-twins}
	Note that the bigraph $T_k$ has no distinct same-side twins, i.e.~no two distinct vertices on the same side have the same neighborhood in the opposite side.
\end{remark}

In order to prove strong EH in the sense of Definition \ref{def: sEH}, we first need  a generalization to arbitrary finitely supported probability measures. Throughout the section, all measures are finitely supported probability measures. If $\mu$ is such a measure, $\supp(\mu)$ denotes its finite support.

\begin{lemma}\label{lem:weighted-pure-pair}
For every $k \geq 2$ and $\sigma\in(0,1]$ there is  $\varepsilon>0$ such that the following holds.  Let $G=(A,B,E)$ be a finite induced $T_k$-free bigraph, and $\mu,\nu$ measures on $A,B$.  If $(\mu\otimes\nu)(E)\leq 1-\sigma$, 
then there are anticomplete sets $A_0\subseteq \supp(\mu)$, $B_0\subseteq \supp(\nu)$ with $\mu(A_0), \nu(B_0)\geq \varepsilon$. 
We may take $\varepsilon=\varepsilon_{T_k}(\sigma/2)$
where $\varepsilon_{T_k}$ is the constant from Fact~\ref{fact:sss}.
\end{lemma}

\begin{proof}
Let $A'\defeq \supp(\mu)$, $B'\defeq \supp(\nu)$, and replace $G$ by the induced sub-bigraph $G(A',B')$.  This sub-bigraph is still induced $T_k$-free, and the value of $(\mu\otimes\nu)(E)$ is unchanged,  so we may assume that $\mu, \nu$ have full support on the \emph{finite} sides $A$ and $B$.

First suppose that $\mu$ and $\nu$ have rational weights and that $(\mu\otimes\nu)(E)\leq 1-\tau$ for some $\tau\in(0,1]$.  So for all $a \in A, b \in B$ we have $\mu(a)=m_a/M, \nu(b)=n_b/N$,
where $m_a,n_b$ are some positive integers and $M=\sum_{a\in A}m_a, N=\sum_{b\in B}n_b$. We construct a finite blow-up $\widetilde G$ of $G$ as follows:   replace each left vertex $a\in A$ by $m_a$ copies $(a,1),\ldots,(a,m_a)$;  replace each right vertex $b\in B$ by $n_b$ copies $(b,1),\ldots,(b,n_b)$; and put an edge between $(a,i)$ and $(b,j)$ exactly when $(a,b)\in E$. The two sides of $\widetilde G$ have sizes $M$ and $N$, and its edge density is
$
\frac{|E(\widetilde G)|}{MN}
=
\sum_{(a,b)\in E}\mu(a)\nu(b)
=
(\mu\otimes\nu)(E)
\leq 1-\tau.
$
And $\widetilde G$ is still induced $T_k$-free (Suppose otherwise.  If two vertices of the induced copy on the same side of $T_k$ projected to the same original vertex of $G$, this would contradict Remark~\ref{lem:no-twins}.  Hence the projection of the copy is injective on each side and gives an induced side-respecting copy of $T_k$ in $G$). Applying Fact~\ref{fact:sss}  to $\widetilde G$ with parameter $\tau$, there are anticomplete sets $\widetilde A_0\subseteq A(\widetilde G), \widetilde B_0\subseteq B(\widetilde G)$ with $|\widetilde A_0|\geq \varepsilon_{T_k}(\tau)M,
|\widetilde B_0|\geq \varepsilon_{T_k}(\tau)N$. 
Let $A_0\subseteq A$ and $B_0\subseteq B$ be the projections of $\widetilde A_0$ and $\widetilde B_0$.  These projected sets are anticomplete in $G$: if some projected pair $(a,b)\in A_0\times B_0$ were an edge of $G$, then every copy of $a$ would be adjacent to every copy of $b$ in $\widetilde G$.  Moreover, $\mu(A_0)=\frac{1}{M}\sum_{a\in A_0}m_a
\geq
\frac{|\widetilde A_0|}{M}
\geq \varepsilon_{T_k}(\tau)$, and similarly $\nu(B_0)\geq \varepsilon_{T_k}(\tau)$.
This proves the rational-weight case.

Now let $\mu$ and $\nu$ be arbitrary probability measures on finite sets $A$ and $B$, still with full support, and assume $(\mu\otimes\nu)(E)\leq 1-\sigma$.
Choose rational valued probability measures $\mu_n,\nu_n$ on $A,B$, also with full support, such that $\mu_n\to\mu$ and $\nu_n\to\nu$ pointwise.  Since $(\alpha,\beta)\longmapsto \sum_{(a,b)\in E}\alpha(a)\beta(b)$
is continuous, for all sufficiently large $n$ we have $(\mu_n\times\nu_n)(E)\leq 1-\sigma/2$.
Applying the rational case with $\tau=\sigma/2$ gives anticomplete sets $A_n\subseteq A$ and $B_n\subseteq B$ with $\mu_n(A_n), \nu_n(B_n)\geq \varepsilon_{T_k}(\sigma/2)$.  Only finitely many pairs of subsets of $A,B$ are possible, so some pair $(A_0,B_0)$ occurs for infinitely many $n$.  It is anticomplete in $G$, and passing to the limit along that subsequence gives $\mu(A_0), \nu(B_0)\geq \varepsilon_{T_k}(\sigma/2)$.
\end{proof}

\begin{lemma}\label{lem:G-Tk-free}
	If $k \geq 2$ and $\varphi(x,y)$ is a $(k,1)$-semi-equation in $M$, then the bipartite graph $G$ that it defines on $M^x \times M^y$ is induced $T_k$-free.
\end{lemma}
\begin{proof}

For $b\in M^y$, let $F_b\defeq \varphi(M;b)$. Assume that $G$ contains an induced side-respecting copy of $T_k$, i.e.~there are $a_0,a_1,\ldots,a_k\in M^x$, $b_1,\ldots,b_k\in M^y$ such that $a_0$ plays the role of $c$, $a_i$ plays the role of $\ell_i$, and $b_i$ plays the role of $r_i$.  Thus $a_0\in F_{b_i}$ for all $i \in [k]$ and $a_i\in F_{b_i}\setminus \bigcup_{j\neq i}F_{b_j}$ for each $i \in [k]$. In particular $\bigcap_{i=1}^k F_{b_i}\neq\varnothing$, and since  $\varphi$ is a $(k,1)$-semi-equation, there are distinct $p,q$ such that $F_{b_p}\subseteq F_{b_q}$.  But the second condition gives $a_p\in F_{b_p}\setminus F_{b_q}$, a contradiction. 
\end{proof}

We are ready to prove Theorem \ref{thm: 1-semieq sEH}:
\begin{proof}[Proof of Theorem \ref{thm: 1-semieq sEH}]

Assume $\varphi(x,y)$ is a $(k,1)$-semi-equation in $M$. We take 

$\eta\defeq \frac{1}{2k}, \ 
\sigma\defeq \eta^2=\frac{1}{4k^2},  \ 
\varepsilon\defeq \varepsilon_{T_k}(\sigma/2)=\varepsilon_{T_k}\!\left(\frac{1}{8k^2}\right), \  \delta_k \defeq \min\left\{
\varepsilon_{T_k}\!\left(\frac{1}{8k^2}\right),
\frac{1-\frac{1}{2k}}{k-1}
\right\}
>0$,
and show that $\varphi(x,y)$ satisfies strong EH with $\delta_k$.

	Let $\mu$ be a finitely supported probability measure on $M^x$, and $\nu$ on $M^y$.  Let $A\defeq \supp(\mu)$, $B\defeq \supp(\nu)$, and consider the finite bigraph $G=(A,B,E)$, where $(a,b)\in E :\Leftrightarrow M \models\varphi(a;b)$. 
For $b\in B$, write $F_b\defeq \varphi(M;b)$, $S_b\defeq A\cap F_b$. We split into two cases according to the $\mu\otimes\nu$-measure of $E$.

\noindent \textbf{Case 1: the relation is not too dense, i.e.~$(\mu\otimes\nu)(E)\leq 1-\sigma$.}

By Lemma~\ref{lem:G-Tk-free}, the bigraph $G$ is induced $T_k$-free.  Applying Lemma~\ref{lem:weighted-pure-pair} gives anticomplete sets $A_0\subseteq A$ and $B_0\subseteq B$ with $\mu(A_0), \nu(B_0) \geq \varepsilon \geq \delta_k$.

\noindent \textbf{Case 2: the relation is very dense, i.e.~$(\mu\otimes\nu)(E)> 1-\sigma$.}

Equivalently, we have $\sum_{b \in B} \mu(A\setminus S_b) \nu(b) <\sigma=\eta^2$. 
Let $B^*\defeq \{b\in B: \mu(A\setminus S_b)\leq \eta\}$, then $\nu(B^*)>1-\eta$.

We next show that the family $\{F_b:b\in B^*\}$ has no $k$-element antichain under inclusion.  Take any $k$ distinct elements $b_1,\ldots,b_k\in B^*$. Since $\mu(A\setminus S_{b_i})\leq \eta$ for each $i$, the union bound gives
\[
\mu\left(A\setminus \bigcap_{i=1}^k S_{b_i}\right) \leq
\sum_{i=1}^k \mu(A\setminus S_{b_i}) \leq k\eta =
\frac{1}{2},
\]
hence
$
\mu\left(\bigcap_{i=1}^k S_{b_i}\right)
\geq \frac{1}{2}$,  so in particular  $\bigcap_{i=1}^k F_{b_i}\neq\emptyset$. The $(k,1)$-semi-equation property gives distinct $p,q \in [k]$ with $F_{b_p}\subseteq F_{b_q}$.

Note that $F_b \subseteq F_{b'}$ is only a pre-order on $B^{\ast}$ (since some of the sets $F_{b}, b \in B^{\ast}$ might be equal). We fix an arbitrary linear order $\leq_0$ on the finite set $B^*$ and define a partial order $\preccurlyeq$ on $B^*$ by
\[
b\preccurlyeq b'
\quad :\Longleftrightarrow\quad
\bigl(F_b\subsetneq F_{b'}\bigr)
\text{ or }
\bigl(F_b=F_{b'}\text{ and } b\leq_0 b'\bigr).
\]
By the previous paragraph, no $k$ distinct elements of $B^*$ form an antichain for $\preccurlyeq$.  Thus the width of the partial order $(B^*,\preccurlyeq)$ is at most $k-1$. By Dilworth's theorem, $B^*$ can be covered by at most $k-1$ chains.  Since $\nu(B^*)>1-\eta$, at least one of these chains, call it $C$, satisfies $\nu(C)>\frac{1-\eta}{k-1}$. 
Let $b_0$ be the $\preccurlyeq$-minimal element of $C$, then $F_{b_0}\subseteq F_b$ for all $b \in C$.
Moreover, as $b_0\in B^*$, we have $\mu(S_{b_0})\geq 1-\eta$.  

Let $A_0\defeq S_{b_0}$ and $B_0\defeq C$. For every $a\in A_0$ and every $b\in B_0$, $a\in F_{b_0}\subseteq F_b$, so  $(a,b)\in E$.  Thus $A_0\times B_0 \subseteq E$, and $\mu(A_0)\geq 1-\eta \geq \delta_k$, $\nu(B_0)>\frac{1-\eta}{k-1} \geq \delta_k$.
\end{proof}

\section{Linear Zarankiewicz in 1-based theories}

\begin{defn}\label{def: lin Zarank} 
	\begin{enumerate}
		\item We say that a partitioned formula $\varphi(x_1, x_2)$  is \emph{combinatorially linear}, if  for every $t \in \mathbb{N}$ there exists some $c = c(\varphi,t) \in \mathbb{R}$ such that: for any finite sets $A_i \subseteq M^{x_i}$ such that the bipartite graph $\varphi(A_1, A_2) = \{(a_1,a_2) \in A_1 \times A_2 : \models \varphi(a_1,a_2)\}$ is $K_{t,  t}$-free (where $K_{t,t}$ denotes the complete bipartite graph with both parts of size $t$), we have $|\varphi(A_1,  A_2)| \leq c (|A_1| + |A_2|)$.
		\item We say that $\varphi(x_1,x_2)$ satisfies  \emph{linear Zarankiewicz} 
		if  for every $t \in \mathbb{N}$ there exists some $c = c(\varphi,t) \in \mathbb{R}$ such that: for any $n \in \mathbb{N}$ and finite sets $A_i \subseteq M^{x_i}$ with $|A_i|=n$ for $i=1,2$ such that $\varphi(A_1, A_2)$ is $K_{t,  t}$-free, we have $|\varphi(A_1,  A_2)| \leq c n$.
		\item And $\varphi(x_1,x_2)$ satisfies  \emph{almost linear Zarankiewicz} 
		if  for every $t \in \mathbb{N}$ and $\varepsilon \in \mathbb{R}_{>0}$  there exists some $c = c(\varphi,t, \varepsilon) \in \mathbb{R}$ such that: for any $n \in \mathbb{N}$ and finite sets $A_i \subseteq M^{x_i}$ with $|A_i|=n$ for $i=1,2$ such that $\varphi(A_1, A_2)$ is $K_{t,  t}$-free, we have $|\varphi(A_1,  A_2)| \leq c n^{1 + \varepsilon}$.

		\item We say that a complete theory $T$ satisfies one of these properties if every partitioned formula does (note that if this holds in one model of $T$, then it holds in all models of $T$).
	\end{enumerate}
\end{defn}
\noindent Note that combinatorial linearity implies linear Zarankiewicz, which in turn implies almost linear Zarankiewicz.

\begin{prop}\label{prop: interp, fields, comb lin}
\begin{enumerate}
\item Almost linear Zarankiewicz implies NIP.
\item If $T_0$ is interpretable in $T$, and $T$ satisfies combinatorial linearity or  (almost) linear Zarankiewicz, then so does $T_0$.
	\item If every partitioned stable formula in $T$ satisfies almost linear Zarankiewicz, then $T$ does not interpret an infinite field.
\end{enumerate}
	
\end{prop}
\begin{proof}
(1) Erd\H{o}s \cite{erdos1964extremal} gives  a probabilistic construction of bipartite graphs not satisfying almost  linear Zarankiewicz (in a strong form). If $T$ is not NIP, every bipartite graph appears as an induced (partite) subgraph of a definable relation.

	(2) It is clear that both properties are preserved under reducts. Assume $T$ satisfies (almost) linear Zarankiewicz, and let $\varphi(y_1, y_2)$ be definable in $M^{\eq}$, with $y_i$ of sort $M^{x_i}/E_{i}$ for a definable equivalence relation $E_i$. So $\varphi(\pi_{E_1}(x_1), \pi_{E_{2}}(x_{2}))$ is equivalent to $\psi(x_1,  x_{2})$ for some $L$-definable $\psi$. Given $A_i \subseteq M/E_{i}, |A_i|=n$, choose a single representative for each element of $A_i$, so we have $A'_i \subseteq M^{x_i}$ with $A_i = \{\pi_{E_i}(a) : a \in A'_i\}, |A'_i|=n$. Then for any $a_i \in A'_i$, $\models \psi(a_1, a_2) \iff \models \varphi(\pi_{E_1}(a_1), \pi_{E_{2}}(a_{2}))$, so the graph induced by $\varphi$ on $A_1 \times A_{2}$ is isomorphic to the graph induced by $\psi$ on $A'_1 \times A'_{2}$. It follows that if $\psi$ satisfies (almost) linear Zarankiewicz, then so does $\varphi$.
	
	(3) Without restricting to stable formulas this was observed in \cite{basit2021zarankiewicz, chernikov2018regularity, chernikov2023semi}, as follows.
	Assume $F$ is an infinite field interpretable in $T$. We let $E(x_1,x_2; y_1,y_2) := y_2 = x_1 \cdot y_1 + x_2$ be the (definable, stable) point-line incidence relation on the affine plane over $F$. If $F$ is of characteristic $0$ it contains $\mathbb{Q}$ as a subfield. The standard witness to the optimal lower bound of Szemer\'edi-Trotter theorem shows that $E$ does not satisfy almost linear Zarankiewicz:  let $A$ be the set of points $\{1, \ldots, n \} \times \{1, \ldots, 2n^2\}$ in $\mathbb{Q}$, and let $B$ be the set of all lines  with slope in $\{1,\ldots, n\}$ and $y$-intercept in $\{1, \ldots, n^2\}$. Then $E$ is $K_{2,2}$-free,  $|A|+|B| = 3n^{3}, |E \cap (A \times B)| = n^{4}$. Now suppose $\ch(F) = p > 0$.
If $T$ is not NIP, by (1) a definable relation witnessing IP does not satisfy almost linear Zarankiewicz, so we assume $T$ is NIP. Then by \cite[Corollary 4.5]{kaplan2011artin}, $F$ contains $\mathbb{F}_p^{\alg}$ as a subfield. Then for an arbitrary $n$, we let $q = p^n$ and, viewing $K := \mathbb{F}_{p^n}$ as a subfield of $F$, let $A$ be the set of all points in $K^2$ and $B$ the set of all lines in $K^2$. Then $|A| = |B| = q^2$ and $|E \cap (A \times B)| = q^3$.

	The following  refinement due to Martin Bays \cite[Question 6.3]{bays2023incidence} shows that in fact the (stable) relation $E$ above fails almost linear Zarankiewicz in any field. Let $F$ be a field. The case $\ch(F) = 0$ is above, so assume 	$\ch(F) = p > 0$. If $F$ does not contain a transcendental element (over the prime field) then $\mathbb{F}_{p^n}$ is a subfield of $F$ for arbitrarily large $n$, and the argument above applies. Otherwise $F$ contains $\mathbb{F}_p[t]$ as a subring. For any $n$, we consider  $A = B := \mathbb{F}_p[t]_{< n} \times \mathbb{F}_p[t]_{<2n}$ (where $\mathbb{F}_p[t]_{< n}$ denotes the set of all polynomials of degree less than $n$, viewed as the corresponding subset of $F$) as the sets of ``points'' and ``lines'', respectively, with the same incidence relation $E$ as above.  Let $N := p^n$, then $|A| = |B| = N^3$, but $|E \cap (A \times B)| = N^4$. This shows that in this case we have a lower bound $n^{4/3}$ for Zarankiewicz. 
\end{proof}

First we note that almost linear Zarankiewicz in one based theories follows quickly from the existing literature.
\begin{defn}(\cite[Definition 2.5]{basit2021zarankiewicz} + \cite[Proposition 2.8]{basit2021zarankiewicz})\label{def: basic rels}
	 A binary relation $E \subseteq X_1 \times X_2$ is basic if there exists a linear order $(S,<)$ and functions $f_i : X_i \to S$ such that $E = \{ (x_1,x_2) \in X_1 \times X_2 : f_1(x_1) < f_2(x_2) \}$.
\end{defn}

\begin{fact}\label{fac: bool basic implies almost lin Zar}
	Let $E(x_1, x_2)$ be a Boolean combination of basic relations. Then $E$ satisfies almost linear Zarankiewicz \cite[Theorem 2.17]{basit2021zarankiewicz}. But it need not satisfy linear Zarankiewicz \cite[Proposition 3.5]{basit2021zarankiewicz}.  (These results are also generalized to hypergraphs in \cite{basit2021zarankiewicz}.)
\end{fact}

\begin{prop}
	If $\varphi(x,y)$ is weakly normal, then it is a Boolean combination of (not necessarily definable) basic relations.
	In particular, if $T$ is $1$-based (or is interpretable in a $1$-based theory), then it satisfies almost linear Zarankiewicz.
\end{prop}
\begin{proof}
	In the notation of the proof of Proposition \ref{prop: weak norm implies strong honest def},  $\varphi(M^x,M^y) \iff \bigcup_{t \in [k]} \{(a,b) : f^{\ell}_t(a) \leq b \leq f^{r}_t(a)\}$, which is a Boolean combination of basic relations.
		The ``in particular'' part follows by Fact \ref{fac: bool basic implies almost lin Zar}, Fact \ref{fac: 1-based iff bool comb weakly normal} and Proposition \ref{prop: interp, fields, comb lin}(2).
\end{proof}

We now strengthen the conclusion to combinatorial linearity for Boolean combinations of weakly normal relations with a different proof. The following is well-known, we include a proof for completeness:
\begin{lemma}
\label{lem: conj of w normal} The family of weakly normal formulas is closed under finite conjunctions. In particular, the family of normal formulas is closed under finite conjunctions.
\end{lemma}

\begin{proof}
It is enough to consider the conjunction of two formulas. Suppose that $\varphi_1(x,y)$ is $\ell_1$-weakly normal and $\varphi_2(x,y)$ is $\ell_2$-weakly normal, and let $\varphi(x,y):=\varphi_1(x,y)\land \varphi_2(x,y)$. We show that $\varphi$ is $\left((\ell_1-1)(\ell_2-1)+1\right)$-weakly normal. Indeed, assume that $\{b_i:i<N\}$ is a finite set of parameters such that $\bigcap_{i<N}\varphi(M,b_i)\neq\emptyset$, where $N=(\ell_1-1)(\ell_2-1)+1$. Fix $a$ in this intersection. For each $\nu\in\{1,2\}$, since all the sets $\varphi_\nu(M,b_i)$ contain $a$, there are at most $\ell_\nu-1$ distinct sets among $\{\varphi_\nu(M,b_i):i<N\}$. Hence there are at most $\left(\ell_1-1\right)\left(\ell_2-1\right)$ possible pairs of sets $
\left(\varphi_1(M,b_i),\varphi_2(M,b_i)\right)$. 
By the pigeonhole principle, for some $i\neq j$ the two corresponding pairs are equal, and then $\varphi(M,b_i)=\varphi(M,b_j)$. This proves $N$-weak normality of $\varphi$. The claim for finite conjunctions follows by induction, and the claim for normal formulas is the special case $\ell_1=\ell_2=2$.
\end{proof}

\begin{lemma}
\label{lem: universal is comb lin}
The full relation $X\times Y$ is combinatorially linear.
\end{lemma}

\begin{proof}
Fix $r<\omega$. If $A\subseteq X$ and $B\subseteq Y$ are finite and $A\times B$ is $K_{r,r}$-free, then $\min\{|A|,|B|\}<r$. Hence $|A\times B|\leq (r-1)(|A|+|B|)$, as required.
\end{proof}

\begin{lemma}
\label{lem: weak normal conj linear} Let $E\subseteq X\times Y$ be given by $E=E_1\land E_2$, where $E_1$ is weakly normal and $E_2$ is combinatorially linear. Then $E$ is combinatorially linear.
\end{lemma}

\begin{proof}
Let $E_1$ be $\ell$-weakly normal. Fix $r<\omega$, and let $A\subseteq X$ and $B\subseteq Y$ be finite such that $E(A,B)$ is $K_{r,r}$-free. Let $c_2:=c(E_2,r)$.

Partition $B$ according to equality of $E_1$-fibers: $b\equiv b'$ if and only if $E_1(X,b)=E_1(X,b')$. For an equivalence class $C$, write $F_C:=E_1(X,b)$ for any $b\in C$. Then
\[
E(A,C)=E_2(A\cap F_C,C).
\]
Moreover $E_2(A\cap F_C,C)$ is $K_{r,r}$-free; otherwise the same $K_{r,r}$ would be contained in $E(A,C)$. Hence, by combinatorial linearity of $E_2$,
\[
|E(A,C)|\leq c_2(|A\cap F_C|+|C|).
\]
Summing over the $\equiv$-classes $C$ in $B$, we obtain
\[
|E(A,B)|\leq c_2\left(\sum_C |A\cap F_C|+|B|\right).
\]
For each $a\in A$, there are at most $\ell-1$ distinct $E_1$-fibers $F_C$ with $a\in F_C$; otherwise $\ell$ distinct such fibers would have a common point, contradicting $\ell$-weak normality. Therefore
\[
\sum_C |A\cap F_C|=\sum_{a\in A}|\{C:a\in F_C\}|\leq (\ell-1)|A|.
\]
Thus $|E(A,B)|\leq c_2((\ell-1)|A|+|B|)$, so $E$ is combinatorially linear.
\end{proof}

\begin{lemma}
\label{lem: lin disj lin}Let $E(x,y)=E_1(x,y)\lor E_2(x,y)$
with both $E_1,E_2$ combinatorially linear. Then $E$ is combinatorially linear.
\end{lemma}

\begin{proof}
Fix $r<\omega$, and let $A\subseteq X,B\subseteq Y$ be finite such that $E(A,B)$ is $K_{r,r}$-free. Then $E_i(A,B)$ is $K_{r,r}$-free for $i=1,2$. Let $c_i:=c(E_i,r)$. Then
\[
|E(A,B)|\leq |E_1(A,B)|+|E_2(A,B)|\leq (c_1+c_2)(|A|+|B|). \qedhere
\]
\end{proof}

\begin{lemma}
\label{lem: comp union weak normal lin}Let
\[
E(x,y):=\neg\bigvee_{1\leq s\leq t}E_s(x,y),
\]
where each $E_s$ is weakly normal. Then $E$ is combinatorially linear.
\end{lemma}

\begin{proof}
We prove this by induction on $t$. For $t=0$ this is Lemma \ref{lem: universal is comb lin}.

Assume $t>0$ and that the result is known for all smaller values of $t$. Let $E_s$ be $\ell_s$-weakly normal. Fix $r<\omega$, and let $A\subseteq X$ and $B\subseteq Y$ be finite such that $E(A,B)$ is $K_{r,r}$-free. If $|A|<r$, then $|E(A,B)|\leq (r-1)|B|$, so we may assume that $|A|\geq r$. Choose $A_*:=\{a_1,\ldots,a_r\}\subseteq A$. Let
\[
B_0:=\{b\in B: \neg E_s(a_i,b) \textrm{ for all } 1\leq s\leq t,1\leq i\leq r\}.
\]
Then $A_*\times B_0\subseteq E$, so $|B_0|<r$. Hence
\[
|E(A,B_0)|\leq (r-1)|A|.
\]

\noindent For $1\leq s\leq t$ and $1\leq i\leq r$, put
\[
B_{s,i}:=\{b\in B:E_s(a_i,b)\}.
\]
Then $B\subseteq B_0\cup\bigcup_{s,i}B_{s,i}$. Partition each $B_{s,i}$ according to equality of $E_s$-fibers. Since every fiber occurring in $B_{s,i}$ contains $a_i$, $\ell_s$-weak normality implies that this partition has at most $\ell_s-1$ classes. Write these classes as $B_{s,i,\lambda}$, and write $F_{s,i,\lambda}:=E_s(X,b)$ for $b\in B_{s,i,\lambda}$. Let
\[
E^{\widehat{s}}(x,y):=\neg\bigvee_{u\neq s}E_u(x,y).
\]
On $(A\setminus F_{s,i,\lambda})\times B_{s,i,\lambda}$, the relation $E_s$ is false, while on $(A\cap F_{s,i,\lambda})\times B_{s,i,\lambda}$ the relation $E$ is false. Hence
\[
E(A,B_{s,i,\lambda})=E^{\widehat{s}}(A\setminus F_{s,i,\lambda},B_{s,i,\lambda}).
\]
The right-hand side is $K_{r,r}$-free, since any $K_{r,r}$ in it would be a $K_{r,r}$ in $E(A,B)$. By the induction hypothesis applied to $E^{\widehat{s}}$, there is a constant $c_{\widehat{s}}=c(E^{\widehat{s}},r)$ such that
\[
|E(A,B_{s,i,\lambda})|\leq c_{\widehat{s}}(|A\setminus F_{s,i,\lambda}|+|B_{s,i,\lambda}|)\leq c_{\widehat{s}}(|A|+|B|).
\]
There are at most $r\sum_{s=1}^t(\ell_s-1)$ triples $(s,i,\lambda)$. Therefore, taking $C:=\max_{1\leq s\leq t}c_{\widehat{s}}$,
\[
|E(A,B)|\leq (r-1)|A|+r\left(\sum_{s=1}^t(\ell_s-1)\right)C(|A|+|B|),
\]
which is linear in $|A|+|B|$.
\end{proof}

\begin{theorem}
\label{thm: bool comb weak norm is comb lin} Let
$E$ be an arbitrary Boolean combination of weakly normal relations. Then
$E$ is combinatorially linear. In particular, if $T$ is $1$-based (or is interpretable in a $1$-based theory), then it is combinatorially linear.
\end{theorem}

\begin{proof}
Write $E$ in disjunctive normal form:
\[
E(x,y)=\bigvee_{q=1}^N\left(P_q(x,y)\land \neg\bigvee_{j=1}^{m_q}N_{q,j}(x,y)\right),
\]
where each $P_q$ is a finite conjunction of weakly normal relations, and each $N_{q,j}$ is weakly normal. By Lemma \ref{lem: lin disj lin}, it is enough to prove that each disjunct is combinatorially linear. So fix $q$. By Lemma \ref{lem: conj of w normal}, $P_q$ is weakly normal (with the convention that the empty conjunction is $X\times Y$, hence normal). By Lemma \ref{lem: comp union weak normal lin}, the relation $\neg\bigvee_{j=1}^{m_q}N_{q,j}$ is combinatorially linear. Hence their conjunction is combinatorially linear by Lemma \ref{lem: weak normal conj linear}. This proves the theorem.
\end{proof}

\section{Combinatorial linearity and strong Erd\H{o}s-Hajnal in Hrushovski's ``ab initio'' constructions}

In this section we consider the same combinatorial questions in the Hrushovski's ``ab initio'' constructions \cite{hrushovski1993new}. Following the notation in \cite{evans2005trivial}, fix integers $r \geq 2$ and $m,n \geq 1$ and let $\mathcal{C}$ be the class of finite $r$-uniform hypergraphs, viewed as structures with a single $r$-ary relation $R(x_1, \ldots, x_r)$ invariant under permutation of coordinates and irreflexive, equipped with the predimension function $\delta(B) = n |B| - m | R(B) |$ and associated self-sufficient embedding $\leq$. 
We let $M_0$ be the infinite rank $\omega$-stable, $\omega$-saturated  uncollapsed generic structure with respect to $(\mathcal{C}_0, \leq)$, and $T_{\textrm{uncol}} := \Th(M_0)$.  We also fix an appropriate function $\mu$  from the set of isomorphism types of minimally simply algebraic extensions in $\mathcal{C}_0$ to the non-negative integers and let  $M := D_{\mu}$ denote the ``collapsed'' strongly minimal generic with respect to the subclass $( \mathcal{C}_{\mu}, \leq )$, and $T_{\textrm{col}} := \Th(M)$ (see e.g.~for the details \cite{zbMATH00663787}).
For most values of $r,m,n$ and appropriate $\mu$, neither $M_0$ nor $M$ is $1$-based (they are only CM-trivial).  In the proof below, we assume some familiarity with the construction, and refer to \cite{zbMATH00663787} for an exposition.

\begin{theorem}\label{thm: ab init lin Zar and sEH}
	Both Hrushovski's collapsed and uncollapsed ab initio structures satisfy linear Zarankiewicz and strong Erd\H{o}s-Hajnal.
\end{theorem}
\begin{proof}
	 Let $M \models T_{\textrm{col}}$ be a model of the collapsed ab initio. We can choose a self-sufficient embedding $f$ of $M$ into a model $M_0 \models T_{\textrm{uncol}}$ of the uncollapsed ab initio. Say that a formula $\varphi(\bar{x}) \in L$ is special if it is of the form $\varphi(\bar{x}) = \exists \bar{y} \psi(\bar{x}, \bar{y})$ with $\psi(\bar{x}, \bar{y})$ quantifier-free and such that if $M \models \psi(\bar{a}, \bar{b}) $ then $\bar{b} \in \sscl(\bar{a})$ (the self-sufficient closure). By the standard back-and-forth for strong substructures, we have that in $T_{\textrm{col}}$  every formula is equivalent to a Boolean combination of special formulas. And if $\varphi(\bar{x})$ is a special formula, then for every tuple $\bar{a}$ in $M$ we have $M \models \varphi(\bar{a}) \iff M_0 \models \varphi(f(\bar{a}))$. Hence, for any formula $\varphi(\bar{x},\bar{y})$ there is  a formula $\varphi'(\bar{x},\bar{y})$ equivalent to $\varphi(\bar{x},\bar{y})$ in $T_{\textrm{col}}$, so that for every  $\bar{a}, \bar{b}$ in $M$ we have $M \models \varphi(\bar{a}, \bar{b}) \iff M_0 \models \varphi(f(\bar{a}), f(\bar{b}))$.
It follows that if every formula satisfies (almost) linear Zarankiewicz (or sEH) in $M_0$, then every formula satisfies (almost) linear Zarankiewicz (respectively, sEH) in $M$.

But by \cite{evans2005trivial} there exists a one-based structure $N_0$ so that $M_0$ is a reduct of $N_0$. It thus follows by Theorem \ref{thm: sEH in 1based} that $M_0$ satisfies strong Erd\H{o}s-Hajnal, and by Theorem \ref{thm: bool comb weak norm is comb lin} that $M_0$ is combinatorially linear.
	\end{proof}

\section{Some questions and future research directions}

\begin{problem}
Generalize these results to definable hypergraphs.
\end{problem}
\noindent E.g.~\cite{basit2021zarankiewicz} establishes an appropriate  generalization of  ``almost linear Zarankiewicz'' for hypergraphs given by Boolean combinations of basic relations.

\begin{problem}\label{prob: 1-semieq dist cell}
	Does every $1$-semi-equational formula admit a distal cell decomposition (in some expansion)? Does every Boolean combination of $1$-semi-equations satisfy almost linear Zarankiewicz?
\end{problem}

\begin{remark}\label{rem: (2,1)-semieq almost lin Zar}
	It was observed in \cite[Proposition 2.23]{chernikov2023semi} that every $(2,1)$-semi-equation is a Boolean combination of  basic relations and admits a distal cell decomposition, but for $k \geq 3$ this remains open. In particular, by Fact \ref{fac: bool basic implies almost lin Zar}, every Boolean combination of $(2,1)$-semi-equations satisfies almost linear Zarankiewicz.
\end{remark}

\begin{problem}
	Does every meet tree $(M, <, \land)$ admit a $1$-semi-equational expansion?
\end{problem}

\begin{remark}
	In an infinitely-branching dense meet tree $\mathcal{T} = (M, <, \land)$, the formula $x < y$ is not a Boolean combination of \emph{definable} $1$-semi-equations \cite[Theorem 3.13]{chernikov2023semi}. But a direct case analysis shows that every partitioned formula is a Boolean combination of $(2,1)$-semi-equations in some expansion (see Lemma 3.11,  proof of Theorem 3.12 and Remark 3.16 in \cite{chernikov2023semi}). Hence $\mathcal{T}$ satisfies almost linear Zarankiewicz (by Remark \ref{rem: (2,1)-semieq almost lin Zar}).
\end{remark}

\begin{problem}
Does every stable $1$-based theory admit a distal expansion? Does the collapsed ab initio admit a distal expansion?
\end{problem}
\noindent We note that the proof of Theorem \ref{thm: ab init lin Zar and sEH} shows that the non-existence of a distal expansion of the collapsed ab initio cannot be witnessed by any of the usual combinatorial obstructions, e.g.~every formula in it admits a (non-definable) distal cell decomposition, etc. It is also open if the collapsed ab initio admits a $1$-based expansion (which would necessarily be of infinite rank, see the discussion in \cite{evans2005trivial}).

 One can also consider other Hrushovski-style constructions, e.g.:
\begin{problem}
	Is  the fusion of two vector spaces combinatorially linear / does it admit a distal expansion?  Many Hrushovski constructions are known to be relatively CM-trivial \cite{blossier2015geometries}; do they also admit relatively one-based expansions?
\end{problem}

\begin{problem}
Does every stable theory satisfying linear Zarankiewicz admit a $1$-based expansion? 
\end{problem}

\begin{problem}\label{prob: almost linear, but not linear T}
	Do there exist stable theories satisfying almost linear Zarankiewicz, but not linear Zarankiewicz?  Strongly minimal theories?
\end{problem}

\noindent We note that this is  possible at the level of a single formula in a stable theory: 
\begin{prop}
	There exists a stable theory $T$ and a formula $\varphi(x,y)$  satisfying almost linear Zarankiewicz, but not linear Zarankiewicz and not strong EH.
\end{prop}
\begin{proof}
We use the example suggested in \cite[Theorem 50]{jiang2020regular}, i.e.~a sequence of Ramanujan graphs with growing degree and girth (Morgenstern's construction \cite{morgenstern1994existence}; we could also use a probabilistic construction as in \cite{brianski2021erdos}).  Namely, for an odd prime power $q$, we take a $(q+1)$-regular Ramanujan graph $G_q$ with $n_q \geq q^{6q}$ vertices and girth $>3q$. As shown in the proof of \cite[Theorem 50]{jiang2020regular}, the family $G_q$ is nowhere dense, hence taking $G$ to be a countable graph given by the disjoint union of $G_q$, $\Th(G)$ is stable, and the edge relation $E(x,y)$ does not satisfy strong EH. 

For $q \geq 3$, girth $> 3q$ implies that $G_q$ has no $4$-cycle, hence viewing $E$ as a binary relation on $G \times G$ it is $K_{2,2}$-free. But $G_q$ is $(q+1)$-regular, so by the Handshaking Lemma $2|E(G_{q})| = (q+1)|V(G_q)|$. As $q$ was unbounded, it follows that $E$ does not satisfy linear Zarankiewicz in $G$.

On the other hand, by the strong VC-density bounds for nowhere dense graphs \cite{pilipczuk2018number}, for every nowhere dense graph $G$, formula $\varphi(\bar{x}, \bar{y})$ and $\delta > 0$ there exists some $c = c(\varphi, \delta)$ so that the shatter function satisfies $\pi_{\mathcal{F}}(n) \leq c n^{|x|+\delta}$, where $\mathcal{F}$ is the family of subsets of $G^{|y|}$ of the form $\varphi(\bar{a},G^{|y|})$ for $\bar{a}$ varying over $G^{|x|}$. Now \cite[Theorem 2.1]{fox2017semi} shows that for every real $c,d$ and integer $t$, if $\pi_{\mathcal{F}}(n) \leq c n^{d}$ and the bipartite graph defined by $\varphi(\bar{x},\bar{y})$ is $K_{t,t}$-free, then for some $c' = c'(c, d, t)$ and any finite $A \subseteq G^{|x|}, B \subseteq G^{|y|}$ with $|A|= m, |B| = n$ we have $|\varphi(A,B)| \leq c' (m n^{1 - 1/d} + n)$ (it is stated there for integer $d$, but the proof goes through unchanged with arbitrary real $d > 1$, possibly increasing the constant). Taking $d := 1 + \delta$ implies almost linear Zarankiewicz for $\varphi(x,\bar{y})$ with $x$ a singleton, in particular for $E(x,y)$.
 \end{proof}

\begin{remark}
	This example, combined with Theorem \ref{thm: 1-semieq sEH}, also shows that nowhere dense graphs need not admit $1$-based, or even $1$-semi-equational expansions.
\end{remark}

\begin{problem}
	Do arbitrary formulas $\varphi(\bar{x},\bar{y})$ in nowhere dense graphs, with both $\bar{x},\bar{y}$ arbitrary finite tuples, satisfy almost linear Zarankiewicz? 
\end{problem}

\begin{problem}
 (``Combinatorial trichotomy conjecture'')	Is there a  strongly minimal theory not satisfying (almost) linear Zarankiewicz and not interpreting an infinite field? Is there a strongly minimal theory with a distal expansion not satisfying the trichotomy?
\end{problem}
\noindent We note that for $o$-minimal $T$, it is proved in \cite{basit2021zarankiewicz} (using the $o$-minimal trichotomy) that almost linear Zarankiewicz holds if and only if there is no infinite definable field.
\begin{problem}
	If all formulas in a theory satisfy strong EH, does it have a distal expansion?
\end{problem}
\noindent An example suggested by Martin Bays  and presented in \cite[Theorem 5.6]{tong2026zarankiewicz} gives a relation $\varphi(x,y)$ satisfying strong EH, but not definable in any distal (or even NTP$_2$, by \cite[Corollary 3.3]{chernikov2015groups}) structure.

\section{Acknowledgements}
We thank Martin Bays, David Evans, Martin Hils, Chris Laskowski, Aris Papadopulos, Pierre Simon and Szymon Toru\'nczyk  for some helpful conversations on the topics of this paper throughout the years. We also thank Institut Henri Poincar\'e in Paris for its hospitality during the ``Model theory, Combinatorics and valued fields'' term in the Spring trimester of 2018. Chernikov was partially supported  by the NSF Research Grant DMS-2246598 and by the Alexander von Humboldt Foundation. Starchenko was partially supported by the Simons Foundation.

\bibliographystyle{plainurl}
\bibliography{ref}

\end{document}